\documentclass{article}

\usepackage[utf8]{inputenc}
\usepackage[english]{babel}
\usepackage{amsfonts}
\usepackage{amssymb}
\usepackage{amsmath}
\usepackage{xypic}
\usepackage[all]{xy}
\usepackage{mathrsfs}
\usepackage{latexsym}
\usepackage{amscd}

\usepackage[mathscr]{eucal}

\usepackage{amsthm}

\usepackage{setspace}                    

\usepackage{tocbibind}

\usepackage{hyperref}

\newtheorem{prop}{Proposition}[section]

\newtheorem{theorem}[prop]{Theorem}

\newtheorem{lemma}[prop]{Lemma}

\newtheorem{corollary}[prop]{Corollary}

\theoremstyle{definition}

\newtheorem{definition}[prop]{Definition}

\theoremstyle{remark}

\newtheorem{remark}[prop]{Remark}

\newtheorem{example}[prop]{Example}

\newtheorem{question}[prop]{Question}

\newcommand{\Sym}{\mathrm{Sym}}

\usepackage{tikz}
\usetikzlibrary{arrows}

\newcommand{\pp}{\mathbb{P}}
\newcommand{\PP}{\mathbb{P}}

\newcommand{\oo}{\mathcal{O}}

\newcommand{\rk}{\mathrm{rk}}

\newcommand{\NN}{\mathbb{N}}
\newcommand{\ZZ}{\mathbb{Z}}

\def\rmapdown#1{\Big\downarrow\rlap{$\vcenter{\hbox{$\scriptstyle
#1$}}$}}
\def\lmapdown#1{\Big\downarrow\llap{$\vcenter{\hbox{$\scriptstyle
#1\;\;\,$}}$}}

\title{Holomorphic symmetric differentials and parallelizable compact complex manifolds}
\author{Ernesto C. Mistretta}
\date{}

\begin{document}

\maketitle

%

\begin{abstract}

We provide a characterization of complex tori using holomorphic symmetric differentials.
With the same method we show that compact complex manifolds of Kodaira dimension 0 having some symmetric power of the cotangent bundle globally generated are quotients of parallelizable manifolds, therefore have an infinite fundamental group.

\end{abstract}

\section{Introduction}

The relations between  differentials and topology of an algebraic variety were known
since the time of K\"ahler and Severi (cf. \cite{kahler},  \cite{severi2}, \cite{severi})

In recent years a lot of progress has been made towards understanding 
the relationship between fundamental group $\pi_1(X)$ of a 
compact K\"ahler manifold $X$
 and 
holomporphic symmetric differentials $H^0(X, Sym^k \Omega^1_X)$,
with $\Omega^1_X$  the holomorphic cotangent bundle of $X$.
In particular it is asked by H\'el\`ene Hesnault whether 
a compact K\"ahler manifold $X$ with  infinite 
fundamental group always carries a non vanishing 
$H^0(X, Sym^k \Omega^1_X)$ for some $k>0$,
and this has an affirmative 
answer, at least in the case where 
the fundamental group has a finite dimensional representaion with infinite image (cf. \cite{brunebarbe}).

On the other hand, 
one could wonder whether the converse is true,
\emph{i.e.} whether a variety with some (or many)
holomporphic symmetric differentials always have an infinite 
fundamental group.
Because of Hodge decomposition
it is immediate to observe that if a compact K\"aler manifold has a non vanishing holomorphic 1-form,
then it has an infinite 
fundamental group.
Also, the presence of particular rank-1 holomorphic symmetric differentials on a projective variety
implies that the fundamental group 
of a projective variety is infinite (cf. \cite{bogomolov}),
and on a complact complex manifold
the presence of a nowhere degenerate holomorphic section 
of $S^2 \Omega_X^1$ as well implies that the 
fundamental group is infinite (cf. \cite{biswasdumitrescu}).

However,
this is not the case in general for higher order symmetric differentials:
there are varieties $X \subset \PP^N$, which are general complete intersections
of high degree in $\PP^N$ and dimension $n \leqslant N/2$,
that have ample cotangent bundle $\Omega^1_X$ 
(cf. \cite{brotbek}) and are simply connected. 
These varieties in particular have 
some symmetric powers of the cotangent bundle with as many holomorphic sections
as possible.
These varieties have 
\emph{semiample} (both weakly and strongly, according to the definitions below)
cotangent bundle and  maximal Kodaira dimension $k(X) = n$,
\emph{i.e.} they are of general type. 
We will show  that this cannot hold in case of a smooth projective variety
of smaller Kodaira dimension.

If $X$ is a  projective variety
of Kodaira dimension $k(X)=0$,
in earlier works inspired by the definition of base loci (cf. \cite{5froci})
and Iitaka fibrations for vector bundles,
we showed that 
having a globally generated symmetric differential bundle 
$Sym^k \Omega^1_X$ for some $k$ 
is equivalent to being isomorphic to an abelian variety, 
and having a generically  generated symmetric 
differential bundle $Sym^k \Omega^1_X$ for some $k$ is equivalent to being birational to an abelian variety (cf. \cite{mistrurbi} and \cite{mistrettaav}). 
In particular in those two cases the fundamental group $\pi_1 (X)$ is infinite.

The purpose of this work is to show the following generalisation 
of these results to the cases of a compact complex manifold,
a compact K\"ahler manifold, and a smooth projective variety:

\begin{theorem}
\label{main}

Let $X$ be a compact complex manifold of dimension $n$ and  Kodaira dimension
$k(X)$.

\begin{enumerate}

\item If $k(X) = 0$ and $\Omega^1_X$ is strongly semiample,
then the fundamental group of $X$ is infinite. 

\item If  $k(X) = 0$ and $X$ is K\"ahler,
then $\Omega^1_X$ is strongly semiample if and only if 
$X$ is biholomorphic to a complex torus.


\item If  $k(X) = 0$ and $X$ is projective,
then $\Omega^1_X$ is weakly semiample if and only if 
$X$ is an \'etale quotient an abelian variety by the action of a finite group.
	
\item If $k(X) <n$, $X$ is projective, and $\Omega^1_X$ is weakly semiample,
then then the fundamental group of $X$ is infinite.

\end{enumerate}

\end{theorem}

\subsection{Acknowledgements}

I am very thankful to Simone Diverio and Andreas H\"oring, for their
ideas and conversations,
and to Sorin Dumitrscu, introducing me to complex parallelizable manifolds.

\section{Notations and basic lemmas}

Let $X$ be a compact complex manifold,
let $E$ be a holomorphic vector bundle on $X$.
Let $\pi \colon \PP(E) \to X$ be the projective bundle of 1-dimensional quotients
of $E$. It comes with a \emph{tautological} quotient
$\pi^* E \twoheadrightarrow \oo_{\PP(E) (1)}$,
where $\oo_{\PP(E)} (1)$ is a 
line bundle on $\PP(E)$.

\begin{definition}
Let $E$ be a holomorphic vector bundle on a compact complex manifold $X$.
\begin{enumerate}

\item We say that the vector bundle $E$ is 
  \emph{strongly semiample} if $\Sym^k E$ is a globally generated vector bundle on $X$, for some $k>0$.
  
\item We say that the vector bundle $E$ is 
  \emph{semiample} or \emph{weakly semiample}  if 
  $\oo_{\PP(E)}(1)^{\otimes k} = \oo_{\PP(E)}(k)$ 
  is a globally generated line bundle on $\PP(E)$, 
  for some $k>0$.
\item  We say that the vector bundle $E$ 
  is \emph{Asymptotically Generically 
  Generated} or \emph{AGG} if there exixsts an open dense subset 
  $U \subseteq X$ such that 
  the map $ev_x \colon H^0(X, \Sym^k )E \to \Sym^k E (x)$
  is surjective for some $k>0$ and for all $x \in U$.

\end{enumerate}

\end{definition}

\begin{remark}

If $L$ is a line bundle on a compact complex variety $X$,
then the Iitaka-Kodaira dimension $k(X,L)$ of $L$ 
is the growth rate of the dimension 
of holomorphic sections $H^0(X, L^{\otimes k})$. 
In particular $k(X, L)=0$ if and only if 
$h^0(X, L^{\otimes k}) \leqslant 1$ for all $k>0$ and it 
is equal to $1$ for some $k>0$.
\end{remark}

The main lemmas we will use do follow 
Fujiwaras constructions in \cite{fujiwara}.

\begin{lemma}
\label{split}

Let $E$ be a holomorphic vector bundle over a compact complex manifold $X$.
Suppose that $E$ admits a morphism $h \colon E \to L$ to a line bundle $L$
such that the induced map
$S^m h \colon \Sym^m E \to L^{\otimes m}$ is  
surjective and  splitting for some $m>0$.
Then $h \colon E \to L$ is surjective and splitting as well.

\end{lemma}

\begin{proof}

First, remark that as $S^m h$ is surjective then $h$ must be surjective.
Let us prove that $h$ splits by recursive induction on $m$. 
Suppose $m \geqslant 2$.
Decompose $S^m h$ as 
$\alpha \circ \beta \colon \Sym^m E \to \Sym^{m-1} E \otimes L \to L^{\otimes m}$,
where 
\[
\alpha = (S^{m-1} h) \otimes 1_L \colon \Sym^{m-1}E \otimes L 
\to L^{\otimes m} ~,
\]
and 
$\beta (v_1 \cdot ... \cdot v_m) 
=\frac{1}{m}\sum (v_1 \cdot ... \cdot \check{v_i} \cdot ... \cdot v_m)\otimes h(v_i) \in \Sym^{m-1}E \otimes L $.

As $S^m h$ splits, then $\alpha$ splits,
and then $S^{m-1} h = \alpha \otimes 1_{L^{-1}}$ splits and we can apply 
recursive induction.

\end{proof}

\section{Parallelizable manifolds}

A complex parallelizable manifold is a complex manifold with trivial cotangent bundle.
It is known that a compact complex manifold is parallelizable if and only if 
it is a quotient of a complex Lie group by a discrete subgroup, 
in particular a compact K\'ahler manifold is parallelizable if an only if it is a torus (cf. \cite{wang}).

\begin{lemma}

Let $X$ be a compact complex manifold admitting a finite étale cover which is parallelizable,
then $\pi_1(X)$ is infinite.

\end{lemma}

\begin{proof}

First remark that by Galois closure any finite \'etale cover $X^{\prime} \to X$ admits a cover 
$X^{\prime \prime} \to X^{\prime}$ such that $X^{\prime \prime} \to X$ is a finite \'etale Galois cover,
furthermore if $X^{\prime}$ is parallelizable then $X^{\prime \prime}$ is parallelizable as well.
So we can suppose that the cover $X^{\prime} \to X$ is a  finite \'etale Galois, and $X^{\prime}$ parallelizable.
Therefore $\pi_1(X^{\prime}) \subseteq \pi_1(X)$ and we need to show that a compact parallelizable manifold 
has infinite fundamental group.

Now suppose that $X^{\prime} = G / \Gamma$ with $G$ a complex Lie group and $\Gamma$  a discrete subgroup.
If $\Gamma$ is finite, then $G$ is compact as well, therefore $G$ is a complex torus and has 
an infinite fundamental group, and $G \to X^{\prime}$ is a finite \'etale covering, so 
$\pi_1(X) \supseteq \pi_1(G)$ is infinite. 
On the other hand if $\Gamma$ is infinite,
then the covering $G \to X^{\prime}$ yields a group extension
$1 \to \pi_1(G) \to \pi_1(X) \to \Gamma \to 1$ therefore $\pi_1(X)$ is infinite.

\end{proof}

\begin{theorem}
\label{trivial}

Let $X$ be a compact complex manifold,
let $E$ be a holomorphic vector bundle on $X$. 
Suppose that $E$ is strongly semiample,
and that its determinant has Iitaka-Kodaira dimension $k(X, \det E)=0$. 
Then there exists a finite Galois cover $f \colon X^{\prime} \to X$ such that 
$f^* E$ is trivial.

\end{theorem}

\begin{proof}

First remark that $\det E$  is a torsion line bundle.
In fact we recall that a line bundle that  has Iitaka-Kodaira dmension $0$
is trivial if globally generated, as it cannot have more than 1-dimensional space of global sections.

As some symmetric power $S^m E$ is globally generated,
for any point $x \in X$ we find sections $\sigma_1, \dots , \sigma_N \in H^0(X, S^m E)$
linearly independent and providing a basis for the fiber $S^m E (x)$, with $N = \rk S^m E$, 
therefore we obtain a section $\sigma_1 \wedge ... \wedge \sigma_N \in H^0(X, (\det E)^{\otimes M})$
of the line bundle $\det (S^m E) = (\det E)^{\otimes M}$
which does not vanish on $x \in X$.
So $(\det E)^{\otimes M}$ is globally generated and of Kodaira dimension $0$,  and therefore  it is trivial.

Then also the symmetric power $S^m E$ is a trivial vector bundle,
in fact if the sections $\sigma_1, \dots , \sigma_N \in H^0(X, S^m E)$ chosen above 
give a basis of $S^m E (x)$  on a point $x \in X$, as the determinant is trivial,
then  they provide a basis at all points $y \in X$, so the induced  map
$\oo_X^{\oplus N} \to S^m E$ is an isomorphism.

Now consider  $\pi \colon \PP(E) \to X$,
 the projective bundle of 1-dimensional quotients
of $E$, and its {tautological} quotient
$\pi^* E \twoheadrightarrow \oo_{\PP(E) (1)}$.
As $S^m E$ is globally generated, then so is $\pi^* S^m E$ and  $\oo_{\PP (E)}(m)$.
Therefore $\oo_{\pp(E)}(m)$ induces a map $\Phi \colon \pp(E) \to \PP^{N-1}$.
Let us show that this map induces many sections of $\pi$.

Let $x \in X$ be a point, and consider the diagram:

\[
\mathop{
\begin{array}{ccccc}
\pp(E(x)) & \stackrel{}{\hookrightarrow}&  \pp(E) & \stackrel{\Phi}{\to} & \pp^N\\
\lmapdown{} & { } & \rmapdown{\pi}\\
\{ x \} & \stackrel{}{\hookrightarrow} & X
\end{array}
}
\]
Now $\Phi_{|\pp(E(x))} \colon \PP(E(x)) \to \pp^N$ is induced by the linear system on 
$\PP(E(x))$ given by the image of the restriction map
\[
H^0(\pp(E), \oo_{\pp(E)}(m)) \to H^0(\pp(E(x)), \oo_{\pp(E(x))}(m)) ~.
\]

But as $S^mE $ is trivial this map is an isomorphism: in fact we have natural isomorphisms
\[
H^0(\pp(E), \oo_{\pp(E)}(m)) \cong H^0(X, S^m E) \cong S^m E(x) ~, \textrm{ and }
\]
\[
H^0(\pp(E(x)), \oo_{\pp(E(x))}(m)) \cong S^m H^0(\pp(E(x)), \oo(1)) \cong S^m E(x) ~.
\]

Therefore the map $\Phi$, when restricted to the projective space
$\PP(E(x))$, is a Veronese embedding, in particular it is injective, 
and letting $x\in X$ vary the map $\Phi_{|\pp(E(x))} \colon \PP(E(x)) \to \pp^N$ has a fixed image.
Therefore for every poiont $w \in \PP^N$, a fiber $\Phi^{-1}(w)$ meets 
a fiber $\pi^{-1}(x) = \PP(E(x))$ exactly in one point,
and actually $\pi$ is trivial as projective bundle.
Therefore  $W = \Phi^{-1}(w)$ provides a section of $\pi \colon \PP(E) \to X$,
with $\pi$ inducing an isomorphism $W$

Such a section yields a quotient $E \to L$, where $L = \oo_{\PP(E)}(1)_{|W}$,
and so $L^{\otimes m} \cong \oo_W$.

By Lemma \ref{split} the vector bundle $E$ splits as $F \oplus L$,
and considering the cyclic \'etale covering $h \colon W^{\prime} \to W$ induced by 
$L^{\otimes m} \cong \oo_W$, then we obtain
on $W^{\prime}$ a splitting $h^* E = h^* F \oplus h^* L = h^* F \oplus \oo_{W^{\prime}}$.
Repeating recursively the argument for the vector bundle $h^* F$ on $W^{\prime}$,
we obtain a finite covering where $E$ becomes trivial.
Then by Galois closure we obtain a finite Galois covering where $E$ becomes trivial.

\end{proof}

\section{Fundamental groups}

From Theorem \ref{trivial} we obtain the first part of Theorem \ref{main}:

\begin{corollary}
\label{corparall}

Let $X$ be a compact complex manifold of Kodaira dimension $k(X) = 0$ 
such that $\Omega^1_X$ is a strongly semiample vector bundle.
Then $X$ admits a finite \'etale Galois covering which is parallelizable.
In particular $\pi_1(X)$ is infinite.

\end{corollary}

\begin{proof}

We just need to apply Theorem \ref{trivial} to the cotangent bundle $\Omega_X$.

\end{proof}

In the compact K\"ahler case we can prove the second part of Theorem \ref{main}.
The proof is actually very similar to the projective case,
which is treated in \cite{mistrurbi} and does generalise easily in this case:

\begin{theorem}
\label{tori}

Let $X$ be a compact K\"ahler manifold of Kodaira dimension $k(X) = 0$ 
such that $\Omega^1_X$ is a strongly semiample vector bundle.
Then $X$ is biholomorphic to a complex torus.

\end{theorem}

\begin{proof}

Let us apply  Theorem \ref{trivial} to the cotangent bundle $\Omega_X$ 
and obtain $\gamma \colon X^{\prime} \to X$ which is an \'etale Galois cover.
Now $X^{\prime}$ is a compact K\"ahler parallelizable manifold,
so it is a torus $T$, and carries a finite group action such that
$\gamma \colon T \to T/G = X$. Now as the covering is \'etale
$\gamma^* \Omega^1_X = \Omega^1_T$ and $\gamma^* S^m \Omega^1_X = S^m \Omega^1_T$.
Therefore 
\[
\gamma^* H^0(X, S^m \Omega^1_X) = H^0(T, S^m \Omega^1_T)^G \subseteq  H^0(T, S^m \Omega^1_T) = S^m H^0(T, \Omega^1_T) ~.
\]

As $S^m \Omega^1_X$ is globally generated, then 
\[
\dim H^0(X, S^m \Omega^1_X) \geqslant \rk S^m \Omega^1_X = \rk S^m \Omega^1_T = \dim S^m H^0(T, \Omega^1_T) ~,
\]
and this implies that $G$ acts trivially on $ S^m H^0(T, \Omega^1_T)$,
and it can be shown that the action of $G$ on $H^0(T, \Omega^1_T)$ must be then through homotheties,
and eventually that it must be trivial on $H^0(T, \Omega^1_T)$ otherwise the action of $G$ on $T$ could not be free
(cf. \cite{mistrettaav}).
Therefore, as $G$ acts trivially on $H^0(T, \Omega^1_T)$, it acts on $T$ by translations,
so the quotient is a torus.

\end{proof}

The third and fourth points in Theorem \ref{main} are consequence  respectively of
the work of Fujiwara  \cite{fujiwara} 
and a recent theorem by Andreas H\"oring \cite{hoering}:

\begin{theorem}[H\"oring]
\label{hoering}
Let $X$ be a projective manifold with strongly semiample cotangent bundle,
i.e. for some positive integer $m \in \NN$ the symmetric product $S^m \Omega^1_X$ is globally generated.
Then there exists a finite cover $X^{\prime} \to X$
such that $X^{\prime} \cong Y \times A$,
where $Y$ has ample canonical bundle and $A$ is an abelian variety.

\end{theorem}

Now, Theorem \ref{hoering} is stated only for projective varieties $X$ 
with strongly semiample cotangent bundle $\Omega^1_X$,
however the result stil holds for varieties with weakly semiample cotangent bundle,
provided that one shows that the canonical bundle $\omega_X = \det \Omega^1_X$ is semiample in this case.
This holds in general for weakly semiample vector bundle on projective varieties,
and is the object of the following theorem, which is proved in \cite{fujiwara}:

\begin{theorem}[Fujiwara]
\label{fuji}

Let $X$ be a projective variety, and let $E$ be a weakly semiample vector bundle on $X$.
Then the determinant $\det E$ is a semiample line bundle.

\end{theorem}

The proof is contained in the work of Fujiwara \cite{fujiwara},
however, we give a detailed proof here,
ad it is related to Theorem \ref{trivial}:

\begin{proof}

Let $\pi \colon \PP(E) \to X$ be the projectivisation of the vector bundle $E$,
and $\oo_{\PP(E)}(1)$ the tautological bundle. 
Fix $x \in X$, we have the fiber $\pi^{-1} (x) = \pp(E(x))$, and 
the the restriction $\oo_{\PP(E)}(1)_{|\pi^{-1} (x)}$ 
is the usual very ample line bundle $\oo_{\PP(E(x))}(1)$.
Let $\xi \in \mathrm{Pic}(\PP(E)) = CH^1 (\PP(E))$, and $r = \rk E$.
Now the Chern classes in the Chow ring of $X$ are determined by the relation:
\[
\sum_{i=0}^r (-1)^i \pi^* c_i(E) \xi^{r-i} =0 ~.
\]
For dimensional reasons we have $\pi_* \xi^k =0$ if $k \leqslant r-2$,
and  $\pi_* \xi^{r-1} = 1 \in CH^{n} (X)$ as $\xi$ restricts to 
$\oo(1)$ on the fibers of $\pi$.
Therefore the formula above gives $\pi_* \xi^r = c_1(E) \in CH^1 (X)$.

Now suppose that $\oo_{\pp(E)}(m)$ is globally generated,
and consider $\Phi \colon \pp(E) \to \PP(H^0(\pp(E), \oo(m))) = \PP^N$.
Notice that, 
as 
$\Phi_{|\pp(E(x))} \colon \pp(E(x)) \to \PP^N$ is induced by the base point free linear system 
\[
\mathrm{Im} (H^0(\PP(E), \oo_{\PP(E)}(m)) \to H^0(\PP(E(x)), \oo_{\PP(E(x))}(m))) ~,
\]
it is a finite map, 
so $N \geqslant r-1$ and therefore we can choose $r$ generic sections 
$\sigma_1, ... , \sigma_r \in H^0(\PP(E), \oo_{\PP(E)}(m))$,
such that $V(\sigma_1) \cap \dots \cap V(\sigma_r) \cap \pp(E(x)) = \emptyset$.
Now the intersection $V(\sigma_1) \cap \dots \cap V(\sigma_r)$ is an effective cycle
whose class is $m^r \xi^r \in CH^r (\pp(E))$.
Therefore the line bundle $(\det E)^{\otimes m^r} = \pi_* (m^r \xi^r)$
has a section which is non zero out of
$\pi (V(\sigma_1) \cap \dots \cap V(\sigma_r) )$,
in particular it does not vanish on $x \in X$.

\end{proof}

In particular the third point of Theorem \ref{main} follows directly from Fujiwara's
characterization in \cite{fujiwara}, 
and the fourth point follows applying H\"oring's theorem above:
suppose $X$ is a smooth projective variety with
$k(X) <n$ and $\Omega^1_X$ weakly semiample.
Then it admits an \'etale finite cover $X^{\prime} \to X$ 
with $X^{\prime} \cong A \times Y$, the variety $Y$ having ample canonical bundle and $A$ an abelian 
variety of dimension $n - k(X)$, therefore 
$\pi(X) \supseteq \pi(A) \times \pi(Y)$ and it is infinite.

\section{Examples, questions, and remarks}

\begin{example}

Given a surjcetive morphism $f \colon Y^{\prime} \to Y$ and a 
vector bundle $E$ on $Y$, then $E$ is weakly semiample if and only if 
$f^* E$ is weakly semiample (cf. \cite{fujita}). However the same cannot be 
said for strongly semiample bundles, even in the case that $f$ is finite \'etale 
and Galois: for example
consider a non-trivial $2$-torsion line bundle $L$ on a curve $C$,
therefore $L$ determines an \'etale $\ZZ / 2\ZZ$-cover 
$f \colon C^{\prime} \to C$ such that $f^* L  = \oo_{C^{\prime}}$. 
Therefore the vector bundle $E = \oo_C \oplus L$ is weakly semiample on $C$ 
but not strongly semiample, as any symmetric power $S^m E$ 
contains a copy of $L$ as direct factor,
and cannot be globally generated.
The vector bundle $f^* E$ being trivial,
it is strongly semiample.

\end{example}

\begin{example} 

According to Theorem \ref{tori} the only
compact K\"ahler manifolds $X$ with strongly semiample 
cotangent bundle and Kodaira dimensio $k(X) =0$ compact tori.

Therefore any \'etale finite quotient of a torus has 
kodaira dimension 0 and weakly semiample cotangent bundle,
if this quotient is not again a torus, it gives an example
of a variety with weakly semiample cotangent bundle but not strongly semiample.

Such an example is any bielliptic surface, which is covered by an abelian surface. 

\end{example}

\begin{remark}
\label{splits}

In the proof of Theorem \ref{trivial} we see that a vector bundle $E$ o a compact complex variety $X$ such that $S^m E$ is trivial splits as a direct sum $F \oplus L$
with $L$ an $m$-torsion line bundle.  As $E = (F \otimes L^{-1} \oplus \oo)\otimes L$
then $S^m E = S^m (F \otimes L^{-1} \oplus \oo)$
has a direct factor  $F \otimes  L^{-1}$ which is trivial as well. Therefore 
$E = L^{\oplus \rk E}$ is a direct sum of the same torsion line bundle.

\end{remark}

\begin{question}

Is there a compact complex manifold $X$ of Kodaira dimension $k(X) =0$
 with cotangent bundle $\Omega^1_X$ which is
not trivial but strongly semiample?
According to Corollary \ref{corparall} and Remark \ref{splits} 
it must be a cyclic quotient of a parallelizable compact manifold,
and cannot be K\"ahler. Futhermore its tangent bundle shoud decompose as a direct sum 
of  isomorphic torsion line bundles. 

\end{question}

\begin{question}

Let $X$ be a compact complex variety, and let $E$ be a \emph{weakly} semiample vector bundle on $X$. 
Is $\det (E)$ a semiample line bundle? And if $X$ is compact K\"ahler?

\end{question}

The techniques used to prove Theorem \ref{fuji} cannot be (directly)
used for the compact K\"ahler case,
nevertheless, in order to apply H\"oring's result 
to a compact K\"ahler manifold in Kodaira
dimension $0$ we would just need that the cotangent bundle be numerically trivial.
We leave these questions to futher investigations.

\begin{question}
\label{numtriv}

Let $X$ be a compact complex variety with a fixed hermitian metric,
and let $E$ be a \emph{weakly} semiample vector bundle on $X$. 
If the Iitaka-Kodaira dimension of the determinant is 
${k}(X, \det (E))=0$,
is $E$ a \emph{numerically trivial} vector bundle? Numerically trivial means
that both $E$ and its dual $E^*$ are nef. 
And what if $X$ is compact K\"ahler?

\end{question}

\begin{remark}

A positive answer to Question \ref{numtriv} above,
would allow to apply H\"oring's theorem to the case of complex tori,
in order to generalize Fujiwara results and
Theorem \ref{main} as follows:
let $X$ be a compact K\"ahler variety
with Kodaira dimension $k(X)=0$,
then the
 cotangent bundle is weakly semiample
if and only if  $X$ is an \'etale quotient of a complex torus by the action of a finite group.

\end{remark}

 \bibliographystyle{amsalpha22}
 \bibliography{parall}

\end{document}